\documentclass[10pt,runningheads,a4paper]{llncs}

\usepackage{graphicx}
\usepackage{amsmath}
\usepackage{amssymb}

\begin{document}
\title{Fault-Tolerant Metric Dimension of $P(n,2)$ with Prism Graph \thanks{This research
is partially supported by The University of Lahor, Pakpattan Campus, Pakistan}
\author{Z. Ahmad$^1$, M. O. Ahmad$^2$, A. Q. Baig$^3$, M. Naeem$^3$}}
\institute{$^1$Department of Mathematics, The University of Lahore, Pakpattan Campus, Pakistan\\
$^2$Department of Mathematics and Statistics, The University of Lahore, Lahore, Pakistan\\
$^3$Department of Mathematics, The University of Lahore, Pakpattan Campus, Pakistan\\
E-mail: zubair786ppn@gmail.com, drchadury@yahoo.com, aqbaig1@gmail.com, naeem@gmail.com}
\vspace{0.4cm}
\authorrunning{Z. Ahmad, M. O. Ahmad, A. Q. Baig, M. Naeem}

\titlerunning{Fault-Tolerant Metric Dimension of $P(n,2)$ with Prism Graph}
\maketitle
\begin{center}2010 \textbf{Mathematics Subject Classification}: 05C12
\end{center}
\begin{abstract}
Let G be a connected graph and $d(a,b)$ be the distance between the vertices $a$ and $b$. A subset $U =\{u_1,u_2,\cdots,u_k\}$ of the vertices is called a resolving set for $G$ if for every two distinct vertices $a,b \in V(G)$, there is a vertex $u_\xi \in U$ such that $d(a,u_\xi)\neq d(b,u_\xi)$. A resolving set containing a minimum number of vertices is called a metric basis for $G$ and the number of vertices in a metric basis is its metric dimension denoted by $dim(G)$. A resolving set $U$ for $G$ is fault-tolerant if $U \setminus \{u\}$ is also a resolving set, for each $u \in U$, and the fault-tolerant metric dimension of $G$ is the minimum cardinality of such a set. In this paper we introduce the study of the fault-tolerant metric dimension of $P(n,2)$ with prism graph.\\
\end{abstract}

\section{Notation and preliminary results}
The length of path  between two vertices $a,b\in V(G)$ in a connected graph is the shortest distance $d(a,b)$ between them. Let $c$ be a vertex of $G$ and $U = \{u_1,u_2,\ldots,u_k\}$ be an ordered set of vertices of $G$. The representation $r(c|U)$ of $c$ with respect to $U$ is the $k$-tuple $(d(c,u_1), d(c,u_2),\ldots, d(c,u_k))$. If distinct vertices of $G$ have distinct representations with respect to $U$, then $U$ is called a resolving set. If the resolving set has minimum number of elements then the set is called a basis for $G$ and the cardinality of $U$ is called the metric dimension of $G$ and is denoted by $dim(G)$. For further study about basis and metric dimension see \cite{10,9,14,6,21,18,7,8}.\\
Generalized Petersen graph $P(n,m)$ is an undirected graph having size $3m$ and order $2n$. It is 3 connected graph which is 3 partite and has independence number 4.
If $U = \{u_1,u_2,\ldots,u_k\}$ is the ordered set of vertices of a graph $G$, then ${\xi}^{th}$ component of $r(c|U)$ is 0 $\Leftrightarrow$ $c= u_\xi$. Thus, in order to show that $U$ is a resolving set it suffices
to verify that $r(a|U)\neq r(b|U)$ for each pair of distinct vertices $a,b\in V(G)\backslash U$.\\
The following property is very useful in order to calculate the dimension of a graph $G$.\\

\begin{lemma} Let $U$ be a resolving set for a connected graph $G$ and $a,b \in
V(G)$. If $d(a,u)=d(b,u)$ for all vertices $u\in
V(G)\setminus\{a,b\}$, then $\{a,\beta\}\cap U\neq\emptyset$.
\end{lemma}
Slater \cite{18} gave the idea of metric dimension in 1975. This concept was applied in chemistry in \cite{4} , this invariant was applied to the navigation of robots in networks in \cite{23}.
The dimension of a path graph in a connected graph $G$ is 1 but if $G$ is a cycle $C_{n}$ with $n\geq 3$ then dimension of $G$ is 2.
\section{$P(n,2)$ with Prism Graph}

\begin{figure}[h]
\begin{center}
  \includegraphics[width=6cm]{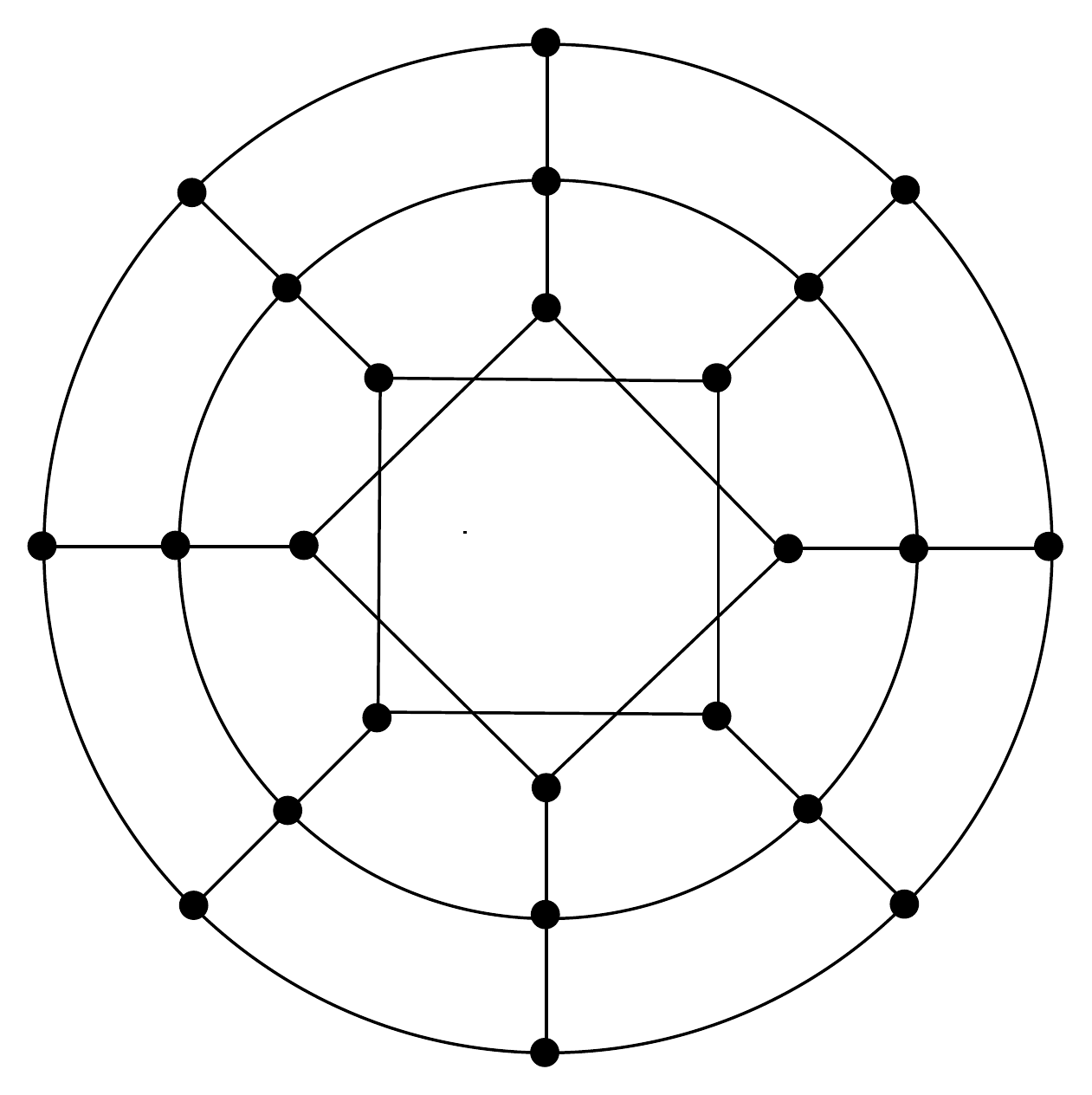}\\
  \caption{$P(8,2)$ with prism graph}\label{3}
  \end{center}
\end{figure}
Let us denote the inner vertices of $G$ by $a_1, a_2,\ldots,a_n$, central vertices by $b_1, b_2,\ldots,b_n$ and outer vertices by $c_1, c_2,\ldots,c_n$. We call the cycle induced by $\{a_{\xi} : 1 \leq \xi\leq n\}$, the inner cycle, the vertices $\{b_{\xi} : 1 \leq \xi\leq n\}$, the middle cycle and the vertices $\{c_{\xi} : 1 \leq \xi\leq n\}$ the outer cycle. We define $V(G)=\{a_{\xi},b_{\xi},c_{\xi} ; 1\leq \xi \leq n\}$ and $E(G)=\{a_{i}b_{\xi+2} ; 1\leq \xi\leq n  \pmod n \}$ $U$ $\{b_{\xi}b_{\xi+1} ; 1\leq \xi\leq n \pmod n \}$ $U$ $\{c_{\xi}c_{\xi+1} ; 1\leq \xi\leq n \pmod n \}$ $U$ $\{a_{\xi}b_{\xi} ; b_{\xi}c_{\xi}; 1\leq \xi\leq n \pmod n \}$


\begin{theorem} For every $n\geq5; n\equiv0,1,2,3(mod 4)$ we have
\begin{center}
  $G=3$.
\end{center}
\end{theorem}
\begin{proof}We will show that only three vertices appropriately chosen suffice to
resolve all vertices. We have the following cases:\\
\textbf{Case 1}: $n\equiv0\textbf{(mod 4)}$\\
For $n = 4k$ with $k>1$ and $k\in\mathbb{Z}^+$. In this case, $\{a_{1}, a_{2}, b_{2k+1}\}$ resolves
$V(G)$. Indeed, $a_{1}$ and $a_{2}$ distinguish the vertices in the inner cycle, middle cycle and outer cycle. To show that $\{a_{1}, a_{2}, b_{2k+1}\}$ resolves vertices of $G$, first we give representations of the vertices in $G$ with respect to $W$.\\
Representations of the inner cycle are\\
$$r(a_{2\xi+1}|W)=\left\{
             \begin{array}{ll}
               (\xi,\xi+2,\xi-1,k+1-\xi), & \hbox{$2\leq \xi\leq k$;} \\
               (2k-\xi,2k+3-\xi,2k+1-\xi,\xi+1-k), & \hbox{$\xi = k+1$;} \\
               (2k-\xi,2k+3-\xi,2k+1-\xi,\xi+1-k), & \hbox{$k+2\leq \xi\leq 2k-1$.}
             \end{array}
           \right.
$$\\
$$r(a_{2\xi}|W)=\left\{
             \begin{array}{ll}
               (\xi+2,\xi-1,\xi+1,k+2-\xi), & \hbox{$2\leq \xi\leq k$;} \\
               (2k+3-\xi,2k+1-\xi,2k+3-\xi,\xi+1-k), & \hbox{$\xi = k+1$;} \\
               (2k+3-\xi,2k+1-\xi,2k+4-\xi,\xi+1-k), & \hbox{$k+2\leq \xi\leq 2k$.}
             \end{array}
           \right.
$$
Representations of the middle cycle are\\
$$r(b_{2\xi+1}|W)=\left\{
             \begin{array}{ll}
               (1,2,2,k+2), & \hbox{$\xi = 0$;} \\
               (\xi+1,\xi+1,\xi,k+2-\xi), & \hbox{$1\leq \xi\leq k-2$;} \\
               (\xi+1,\xi+1,\xi,2), & \hbox{$\xi = k-1$;} \\
               (\xi-1,\xi,\xi,2), & \hbox{$\xi = k+1$;} \\
               (2k+1-\xi,2k+2-\xi,2k+2-\xi,\xi+2-k), & \hbox{$k+2\leq \xi\leq 2k-1$.}
             \end{array}
           \right.
$$\\
$$r(b_{2\xi}|W)=\left\{
             \begin{array}{ll}
               (2,1,2,k+2), & \hbox{$\xi = 1$;} \\
               (\xi+1,\xi,\xi,k+3-\xi), & \hbox{$2\leq \xi\leq k-2$;} \\
               (\xi+1,\xi,\xi,3), & \hbox{$\xi = k-1$;} \\
               (\xi+1,\xi,\xi,1), & \hbox{$\xi = k$;} \\
               (\xi,\xi,\xi+1,1), & \hbox{$\xi = k+1$;} \\
               (\xi-2,\xi-2,\xi,3), & \hbox{$\xi = k+2$;} \\
               (2k+2-\xi,2k+2-\xi,2k+3-\xi,\xi+2-k), & \hbox{$k+3\leq \xi\leq 2k$.}
             \end{array}
           \right.
$$
Representations of the outer  cycle are\\
$$r(c_{2\xi+1}|W)=\left\{
             \begin{array}{ll}
               (2,3,3,k+3), & \hbox{$\xi = 0$;} \\
               (\xi+2,\xi+2,\xi+1,k+3-\xi), & \hbox{$1\leq \xi\leq k-2$;} \\
               (\xi+2,\xi+2,\xi+1,3), & \hbox{$\xi = k-1$;} \\
               (\xi+2,\xi+2,\xi+1,1), & \hbox{$\xi = k$;} \\
               (\xi,\xi+1,\xi+1,3), & \hbox{$\xi = k+1$;} \\
               (2k+2-\xi,2k+3-\xi,2k+3-\xi,\xi+3-k), & \hbox{$k+2\leq \xi\leq 2k-1$.}
             \end{array}
           \right.
$$\\
$$r(c_{2\xi}|W)=\left\{
             \begin{array}{ll}
               (3,2,3,k+4), & \hbox{$\xi = 1$;} \\
               (\xi+2,\xi+1,\xi+1,k+4-\xi), & \hbox{$2\leq \xi\leq k-2$;} \\
               (\xi+2,\xi+1,\xi+1,4), & \hbox{$\xi = k-1$;} \\
               (\xi+2,\xi+1,\xi+1,2), & \hbox{$\xi = k$;} \\
               (\xi+1,\xi+1,\xi+2,2), & \hbox{$\xi = k+1$;} \\
               (\xi-1,\xi-1,\xi,4), & \hbox{$\xi = k+2$;} \\
               (2k+3-\xi,2k+3-\xi,2k+4-\xi,\xi+3-k), & \hbox{$k+3\leq \xi\leq 2k$.}
             \end{array}
           \right.
$$
From the above discussion it follows that  $dim((G) = 4$ in this case.\\
\textbf{Case 2}: $n\equiv1\textbf{(mod 4)}$\\
For $n = 4k+1$ with $k\geq1$ and $k\in\mathbb{Z}^+$. When $n=5$ then $W=\{a_{1}, a_{2}, b_{3}\}$ is the resolving set for $V(G)$. In $V(G)$ the representations of the vertices are\\
$r(a_{4}|W)=(1,1,2,2)$, $r(a_{5}|W)=(2,1,1,2)$, $r(b_{1}|W)=(1,2,2,2)$, $r(b_{2}|W)=(2,1,2,1)$, $r(b_{4}|W)=(2,2,2,1)$, $r(b_{5}|W)=(2,2,2,2)$, $r(c_{1}|W)=(2,3,3,3)$, $r(c_{2}|W)=(3,2,3,2)$, $r(c_{3}|W)=(3,3,2,1)$, $r(c_{4}|W)=(3,3,3,2)$ and $r(c_{5}|W)=(3,3,3,3)$.\\
When $n=9$ then $W=\{a_{1}, a_{2}, b_{5}\}$ is the resolving set for $V(G)$. In $V(G)$ the representations of the vertices are\\
$r(a_{4}|W)=(3,1,3,2)$, $r(a_{5}|W)=(2,3,1,1)$, $r(a_{6}|W)=(2,2,3,2)$, $r(a_{7}|W)=(3,2,2,2)$, $r(a_{8}|W)=(1,3,2,3)$, $r(a_{9}|W)=(3,1,4,3)$, $r(b_{1}|W)=(1,2,2,4)$, $r(b_{2}|W)=(2,1,2,3)$, $r(b_{3}|W)=(2,2,1,2)$, $r(b_{4}|W)=(3,2,2,1)$, $r(b_{6}|W)=(3,3,3,1)$, $r(b_{7}|W)=(3,3,3,2)$, $r(b_{8}|W)=(2,3,3,3)$, $r(b_{9}|W)=(2,2,3,4)$,
$r(c_{1}|W)=(2,3,3,5)$, $r(c_{2}|W)=(3,2,3,4)$, $r(c_{3}|W)=(3,3,2,3)$, $r(c_{4}|W)=(4,3,3,2)$, $r(c_{5}|W)=(4,4,3,1)$, $r(c_{6}|W)=(4,4,4,2)$, $r(c_{7}|W)=(4,4,4,3)$, $r(c_{8}|W)=(3,4,4,4)$ and $r(c_{9}|W)=(3,3,4,5)$.\\
When $n\geq13$ then in this case, $\{a_{1}, a_{2}, b_{2k+1}\}$ resolves
$V(G)$. Indeed, $a_{1}$ and $a_{2}$ distinguish the vertices in the inner cycle, inner pendent cycle, outer pendent cycle and outer cycle. To show that $\{a_{1}, a_{2}, b_{2k+1}\}$ resolves vertices of $G$, first we give representations of the vertices in $G$ with respect to $W$.\\
Representations of the inner cycle are\\
$$r(a_{2\xi+1}|W)=\left\{
             \begin{array}{ll}
               (\xi,\xi+2,\xi-1,k+1-\xi), & \hbox{$2\leq \xi\leq k-1$;} \\
               (\xi,\xi+1,\xi-1,1), & \hbox{$\xi = k$;} \\
               (\xi,\xi-1,\xi-1,2), & \hbox{$\xi = k+1$;} \\
               (2k+3-\xi,2k+1-\xi,2k+3-\xi,\xi+1-k), & \hbox{$\xi = k+2$;} \\
               (2k+3-\xi,2k+1-\xi,2k+4-\xi,\xi+1-k), & \hbox{$k+3\leq \xi\leq 2k$.}
             \end{array}
           \right.
$$\\
$$r(a_{2\xi}|W)=\left\{
             \begin{array}{ll}
               (\xi+2,\xi-1,\xi+1,k+2-\xi), & \hbox{$2\leq \xi\leq k-1$;} \\
               (\xi+1,\xi-1,\xi+1,2), & \hbox{$\xi = k$;} \\
               (\xi-1,\xi-1,\xi,2), & \hbox{$\xi = k+1$;} \\
               (\xi-3,\xi-1,\xi-2,3), & \hbox{$\xi = k+2$;} \\
               (2k+1-\xi,2k+4-\xi,2k+2-\xi,\xi+1-k), & \hbox{$k+3\leq \xi\leq 2k$.}
             \end{array}
           \right.
$$
Representations of the outer pendent cycle are\\
$$r(b_{2\xi+1}|W)=\left\{
             \begin{array}{ll}
               (1,2,2,k+2), & \hbox{$\xi = 0$;} \\
               (\xi+1,\xi+1,\xi,2k-4-\xi), & \hbox{$1\leq \xi\leq k-2$;} \\
               (\xi+1,i+1,\xi,2k-2\xi), & \hbox{$k-1\leq \xi\leq k$;} \\
               (2k+2-\xi,2k+2-\xi,2k+2-\xi,2), & \hbox{$\xi = k+1$;} \\
               (2k+2-\xi,2k+2-\xi,2k+3-\xi,\xi+2-k), & \hbox{$k+2\leq \xi\leq 2k$.}
             \end{array}
           \right.
$$\\
$$r(b_{2\xi}|W)=\left\{
             \begin{array}{ll}
               (2,1,2,2k-4), & \hbox{$\xi = 1$;} \\
               (\xi+1,\xi,\xi,2k-3-\xi), & \hbox{$2\leq \xi\leq k-2$;} \\
               (\xi+1,\xi,\xi,2k+1-2\xi), & \hbox{$k-1\leq \xi\leq k$;} \\
               (\xi,\xi,\xi,1), & \hbox{$\xi = k+1$;} \\
               (\xi-2,\xi-1,\xi-1,3), & \hbox{$\xi = k+2$;} \\
               (2k+2-\xi,2k+3-\xi,2k+3-\xi,\xi+2-k), & \hbox{$k+3\leq \xi\leq 2k$.}
             \end{array}
           \right.
$$
Representations of the outer  cycle are\\
$$r(c_{2\xi+1}|W)=\left\{
             \begin{array}{ll}
               (2,3,3,k+3), & \hbox{$\xi = 0$;} \\
               (\xi+2,\xi+2,\xi+1,2k-3-\xi), & \hbox{$1\leq \xi\leq k-2$;} \\
               (\xi+2,\xi+2,\xi+1,3), & \hbox{$\xi=k-1$;} \\
               (\xi+2,\xi+2,\xi+1,1), & \hbox{$\xi = k$;} \\
               (2k+3-\xi,2k+3-\xi,2k+3-\xi,3), & \hbox{$\xi = k+1$;} \\
               (2k+3-\xi,2k+3-\xi,2k+4-\xi,\xi+3-k), & \hbox{$k+2\leq \xi\leq 2k$.}
             \end{array}
           \right.
$$\\
$$r(c_{2\xi}|W)=\left\{
             \begin{array}{ll}
               (3,2,3,2k-3), & \hbox{$\xi = 1$;} \\
               (\xi+2,\xi+1,\xi+1,2k-2-\xi), & \hbox{$2\leq \xi\leq k-2$;} \\
               (\xi+2,\xi+1,\xi+1,2k+2-2\xi), & \hbox{$k-1\leq \xi\leq k$;} \\
               (\xi+1,\xi+1,\xi+1,2), & \hbox{$\xi = k+1$;} \\
               (\xi-1,\xi,\xi,4), & \hbox{$\xi = k+2$;} \\
               (2k+3-\xi,2k+4-\xi,2k+4-\xi,\xi+3-k), & \hbox{$k+3\leq \xi\leq 2k$.}
             \end{array}
           \right.
$$
From the above discussion it follows that  $dim((G) = 4$ in this case.\\
\textbf{Case 3}: $n\equiv2\textbf{(mod 4)}$\\
For $n = 4k+2$ with $k\geq1$ and $k\in\mathbb{Z}^+$. In this case, $\{a_{1}, a_{2}, b_{2k+2}\}$ resolves
$V(G)$. Indeed, $a_{1}$ and $a_{2}$ distinguish the vertices in the inner cycle, inner pendent cycle, outer pendent cycle and outer cycle. To show that $\{a_{1}, a_{2}, b_{2k+2}\}$ resolves vertices of $G$, first we give representations of the vertices in $G$ with respect to $W$.\\
Representations of the inner cycle are\\
$$r(a_{2\xi+1}|W)=\left\{
             \begin{array}{ll}
               (\xi,\xi+2,\xi-1,k+2-\xi), & \hbox{$2\leq \xi\leq k$;} \\
               (2k+1-\xi,2k+4-\xi,2k+1-\xi,\xi+1-k), & \hbox{$\xi = k+1$;} \\
               (2k+1-\xi,2k+4-\xi,2k+2-\xi,\xi+1-k), & \hbox{$k+2\leq \xi\leq 2k$.}
             \end{array}
           \right.
$$\\
$$r(a_{2\xi}|W)=\left\{
             \begin{array}{ll}
               (\xi+2,\xi-1,\xi+1,k+2-\xi), & \hbox{$2\leq \xi\leq k+1$;} \\
               (2k+4-\xi,2k+2-\xi,2k+5-\xi,\xi-k), & \hbox{$k+2\leq \xi\leq 2k+1$.}
             \end{array}
           \right.
$$
Representations of the middle cycle are\\
$$r(b_{2\xi+1}|W)=\left\{
             \begin{array}{ll}
               (1,2,2,k+3), & \hbox{$\xi = 0$;} \\
               (\xi+1,\xi+1,\xi,k+3-\xi), & \hbox{$1\leq \xi\leq k-2$;} \\
               (\xi+1,\xi+1\xi,,3), & \hbox{$\xi = k-1$;} \\
               (\xi+1,\xi+1,\xi,1), & \hbox{$\xi = k$;} \\
               (\xi,\xi+1,\xi,1), & \hbox{$\xi = k+1$;} \\
               (\xi-2,\xi-1,\xi-1,3), & \hbox{$\xi = k+2$;} \\
               (2k+2-\xi,2k+3-\xi,2k+3-\xi,\xi+2-k), & \hbox{$k+3\leq \xi\leq 2k$.}
             \end{array}
           \right.
$$\\
$$r(b_{2\xi}|W)=\left\{
             \begin{array}{ll}
               (2,1,2,k+2), & \hbox{$\xi = 1$;} \\
               (\xi+1,\xi,\xi,k+3-\xi), & \hbox{$2\leq \xi\leq k-1$;} \\
               (\xi+1,\xi,\xi,2), & \hbox{$\xi = k$;} \\
               (\xi-1,\xi-1,\xi,2), & \hbox{$\xi = k+2$;} \\
               (2k+3-\xi,2k+3-\xi,2k+4-\xi,\xi+1-k), & \hbox{$k+3\leq \xi\leq 2k+1$.}
             \end{array}
           \right.
$$
Representations of the outer  cycle are\\
$$r(c_{2\xi+1}|W)=\left\{
             \begin{array}{ll}
               (2,3,3,k+4), & \hbox{$\xi = 0$;} \\
               (\xi+2,\xi+2,\xi+1,k+4-\xi), & \hbox{$1\leq \xi\leq k-2$;} \\
               (\xi+2,\xi+2,\xi+1,4), & \hbox{$\xi = k-1$;} \\
               (\xi+2,\xi+2,\xi+1,2), & \hbox{$\xi = k$;} \\
               (\xi+1,\xi+2,\xi+1,2), & \hbox{$\xi = k+1$;} \\
               (\xi-1,\xi,\xi,4), & \hbox{$\xi = k+2$;} \\
               (2k+3-\xi,2k+4-\xi,2k+4-\xi,\xi+3-k), & \hbox{$k+3\leq \xi\leq 2k$.}
             \end{array}
           \right.
$$\\
$$r(c_{2\xi}|W)=\left\{
             \begin{array}{ll}
               (3,2,3,k+3), & \hbox{$\xi = 1$;} \\
               (\xi+2,\xi+1,\xi+1,k+4-\xi), & \hbox{$2\leq \xi\leq k-1$;} \\
               (\xi+2,\xi+1,\xi+1,3), & \hbox{$\xi = k$;} \\
               (\xi+2,\xi+1,\xi+1,1), & \hbox{$\xi = k+1$;} \\
               (\xi,\xi,\xi+1,3), & \hbox{$\xi = k+2$;} \\
               (2k+4-\xi,2k+4-\xi,2k+5-\xi,\xi+2-k), & \hbox{$k+3\leq \xi\leq 2k+1$.}
             \end{array}
           \right.
$$

From the above discussion it follows that  $dim((G) = 4$ in this case.\\
\textbf{Case 4}: $n\equiv3\textbf{(mod 4)}$\\
For $n = 4k+3$ with $k\geq1$ and $k\in\mathbb{Z}^+$. When $n=7$ then $W=\{a_{1}, a_{2}, b_{4}\}$ is the resolving set for $V(G)$. In $V(G)$ the representations of the vertices are\\
$r(a_{4}|W)=(2,1,3,1)$, $r(a_{5}|W)=(2,2,1,2)$, $r(a_{6}|W)=(1,2,2,2)$, $r(a_{7}|W)=(3,1,2,3)$,
$r(b_{1}|W)=(1,2,2,3)$, $r(b_{2}|W)=(2,1,2,2)$, $r(b_{3}|W)=(2,2,1,1)$,$r(b_{5}|W)=(3,3,2,1)$, $r(b_{6}|W)=(2,3,3,2)$, $r(b_{7}|W)=(2,2,3,3)$, $r(c_{1}|W)=(2,3,3,4)$, $r(c_{2}|W)=(3,2,3,3)$, $r(c_{3}|W)=(3,3,2,2)$, $r(c_{4}|W)=(4,3,3,1)$, $r(c_{5}|W)=(4,4,3,2)$, $r(c_{6}|W)=(3,4,4,3)$ and $r(c_{7}|W)=(3,3,4,4)$.\\
When $n=11$ then $W=\{a_{1}, a_{2}, b_{6}\} $ is the resolving set for $V(G)$. In $V(G)$ the representations of the vertices are\\
$r(a_{4}|W)=(4,1,3,2)$, $r(a_{5}|W)=(2,4,1,2)$, $r(a_{6}|W)=(3,2,4,1)$, $r(a_{7}|W)=(3,3,2,2)$, $r(a_{8}|W)=(2,3,3,2)$, $r(a_{9}|W)=(4,2,3,3)$, $r(a_{10}|W)=(1,4,2,3)$, $r(a_{11}|W)=(3,1,4,4)$, $r(b_{1}|W)=(1,2,2,5)$, $r(b_{2}|W)=(2,1,2,4)$, $r(b_{3}|W)=(2,2,1,3)$, $r(b_{4}|W)=(3,2,2,2)$, $r(b_{5}|W)=(3,3,2,1)$, $r(b_{7}|W)=(4,4,3,1)$, $r(b_{8}|W)=(3,4,4,2)$, $r(b_{9}|W)=(3,3,4,3)$, $r(b_{10}|W)=(2,3,3,4)$, $r(b_{11}|W)=(2,2,3,5)$, $r(c_{1}|W)=(2,3,3,6)$, $r(c_{2}|W)=(3,2,3,5)$, $r(c_{3}|W)=(3,3,2,4)$, $r(c_{4}|W)=(4,3,3,3)$, $r(c_{5}|W)=(4,4,3,2)$, $r(c_{6}|W)=(5,4,4,1)$, $r(c_{7}|W)=(5,5,4,2)$, $r(c_{8}|W)=(4,5,5,3)$, $r(c_{9}|W)=(4,4,5,4)$, $r(c_{10}|W)=(3,4,4,5)$ and $r(c_{11}|W)=(3,3,4,6)$.\\
When $n\geq15$ then in this case, $\{a_{1}, a_{2}, b_{2k+2}\}$ resolves
$V(G)$. Indeed, $a_{1}$ and $a_{2}$ distinguish the vertices in the inner cycle, inner pendent cycle, outer pendent cycle and outer cycle. To show that $\{a_{1}, a_{2}, b_{2k+2}\}$ resolves vertices of $G$, first we give representations of the vertices in $G$ with respect to $W$.\\

Representations of the inner cycle are\\
$$r(a_{2\xi+1}|W)=\left\{
             \begin{array}{ll}
               (\xi,\xi+2,\xi-1,k+2-\xi), & \hbox{$2\leq \xi\leq k$;} \\
               (\xi,\xi,\xi-1,2), & \hbox{$\xi = k+1$;} \\
               (2k+4-\xi,2k+2-\xi,2k+4-\xi,\xi+1-k), & \hbox{$\xi = k+2$;} \\
               (2k+4-\xi,2k+2-\xi,\xi+1-k), & \hbox{$k+3\leq \xi\leq 2k$.}
             \end{array}
           \right.
$$\\
$$r(a_{2\xi}|W)=\left\{
             \begin{array}{ll}
               (\xi+2,\xi-1,\xi+1,k+2-\xi), & \hbox{$2\leq \xi\leq k$;} \\
               (\xi,\xi-1,\xi,1), & \hbox{$\xi = k+1$;} \\
               (\xi-2,\xi-1,\xi-1,2), & \hbox{$\xi = k+2$;} \\
               (2k+2-\xi,2k+5-\xi,2k+3-\xi,\xi-k), & \hbox{$k+3\leq \xi\leq 2k+1$.}
             \end{array}
           \right.
$$
Representations of the outer pendent cycle are\\
$$r(b_{2\xi+1}|W)=\left\{
             \begin{array}{ll}
               (1,2,2,k+3), & \hbox{$\xi = 0$;} \\
               (\xi+1,\xi+1,\xi,k+3-\xi), & \hbox{$1\leq \xi\leq k-1$;} \\
               (\xi+1,\xi+1,\xi,1), & \hbox{$k\leq \xi\leq k+1$;} \\
               (\xi-1,\xi-1,\xi,3), & \hbox{$\xi = k+2$;} \\
               (2k+3-\xi,2k+3-\xi,2k+4-\xi,\xi+2-k), & \hbox{$k+3\leq \xi\leq 2k+1$.}
             \end{array}
           \right.
$$\\
$$r(b_{2\xi}|W)=\left\{
             \begin{array}{ll}
               (2,1,2,k+2), & \hbox{$\xi = 1$;} \\
               (\xi+1,\xi,\xi,k+3-\xi), & \hbox{$2\leq \xi\leq k-1$;} \\
               (\xi+1,\xi,\xi,2), & \hbox{$\xi = k$;} \\
               (\xi-1,\xi,\xi,2), & \hbox{$\xi = k+2$;} \\
               (2k+3-\xi,2k+4-\xi,2k+4-\xi,\xi+1-k), & \hbox{$k+3\leq \xi\leq 2k+1$.}
             \end{array}
           \right.
$$
Representations of the outer  cycle are\\
$$r(c_{2\xi+1}|W)=\left\{
             \begin{array}{ll}
               (2,3,3,k+4), & \hbox{$\xi = 0$;} \\
               (\xi+2,\xi+2,\xi+1,k+4-\xi), & \hbox{$1\leq \xi\leq k-1$;} \\
               (\xi+2,\xi+2,\xi+1,2), & \hbox{$k\leq \xi\leq k+1$;} \\
               (\xi,\xi,\xi+1,4), & \hbox{$\xi = k+2$;} \\
               (2k+4-\xi,2k+4-\xi,2k+5-\xi,\xi+3-k), & \hbox{$k+3\leq \xi\leq 2k+1$.}
             \end{array}
           \right.
$$\\
$$r(c_{2\xi}|W)=\left\{
             \begin{array}{ll}
               (3,2,3,k+3), & \hbox{$\xi = 1$;} \\
               (\xi+2,\xi+1,\xi+1,k+4-\xi), & \hbox{$2\leq \xi\leq k-1$;} \\
               (\xi+2,\xi+1,\xi+1,3), & \hbox{$\xi = k$;} \\
               (\xi+2,\xi+1,\xi+1,1), & \hbox{$\xi = k+1$;} \\
               (\xi,\xi+1,\xi+1,3), & \hbox{$\xi = k+2$;} \\
               (2k+4-\xi,2k+5-\xi,2k+5-\xi,\xi+2-k), & \hbox{$k+3\leq \xi\leq 2k+1$.}
             \end{array}
           \right.
$$
From the above discussion it follows that  $dim((G) = 4$ in this case.\\

\end{proof}

\section{Relation between $\beta(G)$ and $\acute{\beta}(G)$}
In this section we prove that fault-tolerant metric dimension is bounded by a function of
metric dimension (independent of the graph). As usual, $N(v)$ and $N[v]$ denote vertex $v’s$
open and closed neighborhoods, respectively.
\begin{lemma}
Let $W$ be a resolving set of $G$. For each vertex $w\in W$, let $T(w)= \{x \in V(G):N(w)\subseteq N(x)\}$. Then $\acute{W}= \bigcup_{w \in W}(N[w] \bigcup T(w))$ is a fault-tolerant resolving set of $G$.
\end{lemma}
\begin{proof}
Consider a vertex $w \in \acute{W}$.
If $w \notin {W}$ then $\acute{W} \setminus \{w\}$ resolves $G$ since $W \subseteq \acute{W}\setminus \{w\}$.
Now assume that $w \in W$.\\
Let $m$ and $n$ be distinct vertices of $G$. We must show that  $m$ and $n$ are resolved by some
vertex in $\acute{W}\setminus \{w\}$. If not, then $w$ must resolve $m$ and $n$  since $W$ resolves $m$ and $n$ . Without
loss of generality $d(w,m) = d(w,n)-1$.\\
First suppose that $m\neq w$. Let $x$ be the neighbour of $w$ on a shortest path between $w$
and $m$. Then $x \in \acute{W}\setminus \{w\}$ and $d(w,m) = d(x,m)+1$. Thus $d(x,m)+1\leq d(w,n)-1$. Now
$d(x,n)+1\leq d(x,n)$. Hence $d(x,m)+1\leq d(w,n)$. Thus $x$ resolves $m$ and $n$.\\
Now assume that $m=w$. If $n \notin \acute{W}$ then  $n \in \acute{W}\setminus \{w\}$ and $n$ resolves $m$ and $n$. Otherwise
$d(w,n)\geq 2$ and $n$ is not adjacent to some neighbour $x$ of $m$. Then $d(w,x) = 1$ and
$d(n,x)\geq 2$. Thus $x \in \acute{W}\setminus \{w\}$ resolves $w (= m)$ and $n$.
Hence $\acute{W}$ is a fault-tolerant resolving set of $G$.
\end{proof}
\begin{lemma}
Let $W$ be a resolving set in a graph $G$. Then for each vertex $w \in \acute{W}$, the
number of vertices of $G$ at distance at most $k$ from $w$ is at most $1 + k(2k + 1)^{|W|-1}$.
\end{lemma}
\begin{proof}
Let $1\leq d(w,x)\leq k$. For every vertex $u \in \acute{W}$ with $u\neq w$, we have $|d(x,u)-d(x,w)|\leq k$. Thus there are $2k + 1$ possible values for $d(x,u)$ and there are at most $k$ possible values for $d(x,w)$. Thus the vector of distances from $x$ to $W$ has $1 + k(2k + 1)^{|W|-1}$ possible values. The result follows, since the vertices at distance at most $k$ are resolved by $W$.
\end{proof}
\begin{theorem}
Fault-tolerant metric dimension is bounded by a function of the metric
dimension (independent of the graph). In particular, $\acute{\beta}(G)\leq \beta(G)(1 + 2.5^{\beta(G)-1})$ for every graph $G$.
\end{theorem}
\begin{proof}
  Let $W$ be a metric basis for a graph $G$. Lemma 3 with $k = 2$ implies that
$|N[w]\bigcup T(w)|\leq 1 + 2.5^{\beta(G)-1})$ for each vertex $w \in W$. Thus $\acute{\beta}(G)\leq \beta(G)(1 + 2.5^{\beta(G)-1})$.
\end{proof}

\section{Conclusion}
The concept of resolving set and that of metric dimension is defined in monograph. Graphs are special examples of metric spaces with their intrinsic path metric. In this paper the metric dimension of the $G$, graph has been studied and  proved that graph has a constant metric dimension 3.\\

\end{document}